\documentclass[12pt,psamsfonts]{amsart}
\usepackage[utf8]{inputenc}
\usepackage{amsmath}
\usepackage{amsthm}
\usepackage{amssymb}
\usepackage{amscd}
\usepackage{amsfonts}
\usepackage{amsbsy}
\usepackage{graphicx}
\usepackage[dvips]{psfrag}
\usepackage{array}
\usepackage{color}
\usepackage{xcolor}
\usepackage{epsfig}
\usepackage{url}
\usepackage{soul}
\usepackage[foot]{amsaddr}
	
\usepackage{overpic}
\usepackage{epstopdf}
\usepackage[pagewise]{lineno}%\linenumbers

\newcommand{\sgn}{\mathrm{sign}}

\newcommand{\Int}{\mathrm{Int}}

\newtheorem {proposition} {Proposition}
\newtheorem {corollary} {Corollary}

\newtheorem {remark} {Remark}

\begin{document}
\renewcommand{\arraystretch}{1.5}

\title[Properties of the Poincar\'{e} half-maps]
{Properties of  Poincar\'{e} half-maps for planar linear systems and some direct applications to periodic orbits of piecewise systems}

\author[V. Carmona, F. Fern\'andez-S\'{a}nchez, E. Garc\'{\i}a-Medina and D. D. Novaes]
{Victoriano Carmona$^{1},$ Fernando Fern\'andez-S\'{a}nchez$^2,$ Elisabeth Garc\'{\i}a-Medina$^3$  and Douglas D. Novaes$^{\ast 4}$}

\address{$^{1}$ Dpto. Matem\'{a}tica Aplicada II \& IMUS, Universidad de Sevilla, Escuela Polit\'ecnica Superior.
Calle Virgen de \'Africa 7, 41011 Sevilla, Spain.} 
\email{vcarmona@us.es}

\address{$^2$ Dpto. Matem\'{a}tica Aplicada II \& IMUS, Universidad de Sevilla, Escuela T\'{e}cnica Superior de Ingenier\'{i}a.
Camino de los Descubrimientos s/n, 41092 Sevilla, Spain.} 
\email{fefesan@us.es} 

\address{$^3$ Dpto. Matem\'{a}tica Aplicada II, Universidad de Sevilla, Escuela Polit\'ecnica Superior.
Calle Virgen de \'Africa 7, 41011 Sevilla, Spain.} 
\email{egarme@us.es}

\address{$^{\ast 4}$ Departamento de Matem\'{a}tica - Instituto de Matem\'{a}tica, Estat\'{i}stica e Computa\c{c}\~{a}o Cient\'{i}fica (IMECC) - Universidade
Estadual de Campinas (UNICAMP), \ Rua S\'{e}rgio Buarque de Holanda, 651, Cidade Universit\'{a}ria Zeferino Vaz, 13083-859, Campinas, SP,
Brazil.} \email[Corresponding author]{ddnovaes@unicamp.br}



\maketitle

\begin{abstract}
This paper deals with fundamental properties of Poincar\'e half-maps defined on a straight line for planar linear systems. Concretely, we focus on the analyticity of the Poincar\'e half-maps, their series expansions (Taylor and Newton-Puiseux) at the tangency point and at infinity, the relative position between the graph of Poincar\'e half-maps and the bisector of the fourth quadrant, and the sign of their second derivatives. All these properties are essential to understand the dynamic behavior of planar piecewise linear systems. Accordingly, we also provide some of their most immediate, but non-trivial, consequences regarding periodic orbits.
\end{abstract}

\newpage

\section{Introduction}\label{sec:Introduction}
The study of the qualitative properties of distinguished solutions of piecewise linear differential systems rests mainly on the analysis of the features of Poincar\'e maps, which are defined as composition of transition maps between the separation manifolds. Sometimes these transition maps are called Poincar\'e half-maps. The linearity of the system in each zone invites to its integration, which automatically causes the emergence of a wide range of cases due to the nature of the different spectra of the matrices of the linear systems and the relative position between the equilibria, if any, and the separation manifolds. The number of cases to be studied is high even for planar systems with two zones of linearity. Moreover, the direct integration leads to different nonlinear equations where the flight time appears as a new variable. 

Since the publication of the seminal work by Freire et al. \cite{FreireEtAl98}, a large number of interesting papers have appeared in order to establish the dynamical behavior in planar piecewise linear systems with two zones of linearity and, in particular, to give conditions for the existence and stability of limit cycles and to provide an optimal bound for the number of coexisting limit cycles (see for instance, \cite{FreireEtAl2013,Huan:2021ti,HuanYang2013,HuanYang2014,LI2021101045,LiLLibre2020, LlibreEtAl15, LlibreEtAl2008}). None of these papers considers all the possible cases. Moreover, they are forced to use individualized approaches to study the different kind of functions that arise due to the distinct spectra of the matrices. This causes that a same result is usually expressed in different terms and, sometimes, it may be a hard task to obtain a common and brief statement for it. Thus, the use of individualized techniques for each case does not allow a unified view of several properties of the Poincar\'e half-maps and, when it does, more effort is required to complete the case-by-case study and to achieve independent statements of these cases. 

This paper relies on a novel characterization of Poincar\'e half-maps for planar linear systems \cite{CARMONA2021319} which allows us to see the properties of these maps from a common point of view and to prove the results in a simple way, without the need of making particularized case-by-case studies. Accordingly, we will not have any of the disadvantages mentioned in the previous paragraphs because this novel characterization does not require integration of the systems and, therefore, the distinction of the spectra of the matrices is not needed. The strength of this approach can be seen in \cite{Carmona21}, where the uniqueness of limit cycles for continuous piecewise linear systems was provided in a simple and synthesized way.

In the framework of the study of Poincar\'e half-maps for planar linear systems, the most relevant properties are those related to the local behavior at tangency points between the flow and the Poincar\'e section, the behavior at infinity (obviously, in the case of the focus or center), and the sign of the derivatives. Some of these properties have been proven just for concrete cases. Even those which are valid for all situations have been proven in a large case-by-case study. This manuscript is primarily devoted to simplifying and unifying the proofs of these properties by considering all possible scenarios simultaneously. In addition, it will be stated an interesting fact about the relative position between the graphs of Poincar\'e half-maps and the bisector of the fourth quadrant. Among other things, from this property it is direct that Poincar\'e half-maps inherit the expansion/compression behavior of the flow of the planar linear system. Additionally, it is proven here that this relative position  is also related to the (constant) sign of its second derivative.

As might be expected from the first two paragraphs of this introduction, all the previously commented properties have direct applications to planar piecewise linear systems; from the analysis of stability and bifurcations of equilibria, singularities, or the infinity, to the existence and characterization of periodic orbits and the obtention of uniform bounds to the number of limit cycles. In this work, some straightforward conclusions concerning the periodic behavior are obtained.

The paper is organized as follows. Section \ref{sec:char} presents the integral characterization of Poincar\'e half-maps for planar linear systems given in \cite{CARMONA2021319}. Two basic consequences of this characterization for the Poincar\'e half-maps are their analyticity and their understanding as solutions of a differential equation. In Section \ref{ssec:orin-inf}, we summarize the results on analyticity of the Poincar\'e half-maps given in \cite{CARMONA2021319} and obtain the Taylor and Newton-Puiseux series expansions at tangency points and infinity by means of the differential equation. Section \ref{ssec:relativeposition} studies the relationship between the graphs of Poincar\'e half-maps and the bisector of the fourth quadrant, which is used to establish, in Section \ref{sec:derivadasegunda}, the sign of the second derivatives of the Poincar\'e half-maps. Finally, Section \ref{sec:consequences} addresses the analysis of the periodic behavior of planar piecewise linear systems with two zones separated by a straight line. There, some direct consequences of the properties of Poincar\'e half-maps obtained in previous sections are stated.
 
 \section{Integral characterization for the Poincar\'{e} half-maps}\label{sec:char}

Let us consider, for $\mathbf{x}=(x_1,x_2)^T$, the autonomous linear system 
\begin{equation}
\label{eq:sislinnohom}
    \dot{\mathbf{x}}=A\, \mathbf{x}+\mathbf{b}
\end{equation}
where $A=(a_{ij})_{i,j=1,2}$ is a real matrix and $\mathbf{b}=(b_1,b_2)^T\in\mathbb{R}^2$. Let us choose, without loss of generality, the Poincar\'e section $\Sigma=\{(x_1,x_2)\in \mathbb{R}^2:x_1=0\}$.

Notice that if the coefficient $a_{12}$ vanishes, system (\ref{eq:sislinnohom}) is uncoupled and a Poincar\'e half-map on section $\Sigma$ cannot be defined. 
Hence, let us assume in this work that $a_{12}\ne0$ (observability condition \cite{CarmonaEtAl02}). On the one hand, observe that, among other configurations, this condition removes the possibility of star-nodes to appear. On the other hand, under the assumption $a_{12}\neq0$, the linear change of variable $x=x_1$, 
$
y=a_{22} x_1-a_{12} x_2-b_1,
$
with $a=a_{12}b_2-a_{22}b_1$, allows to write system \eqref{eq:sislinnohom}  into the generalized Li\'{e}nard form,
 \begin{equation}\label{lienard}
\left(\begin{array}{l} \dot{x}\\ \dot{y} \end{array}\right)=\left(\begin{array}{rr} T & -1\\
D& 0\end{array}\right) \left(\begin{array}{l} x\\ y \end{array}\right)-\left(\begin{array}{l}0\\a\end{array}\right),
\end{equation}
where $T$ and $D$ stand for the trace and the determinant of matrix $A$, respectively. 
 In the new coordinates, since $x_1=x$, Poincar\'e section $\Sigma$ remains the same.

The first equation of system \eqref{lienard} evaluated on the section $\Sigma$ is reduced to $\dot{x}|_{\Sigma}=-y$. Therefore, the flow of the system
crosses $\Sigma$ from the half-plane $\Sigma^+=\{(x,y)\in \mathbb{R}^2:x>0\}$ to the half-plane $\Sigma^-=\{(x,y)\in \mathbb{R}^2:x<0\}$ when $y>0$, from $\Sigma^-$ to $\Sigma^+$ when 
$y<0$, and it is tangent to $\Sigma$ at the origin.

Since this work is devoted to Poincar\'e half-maps of system \eqref{lienard} corresponding to the section $\Sigma$ and due to the fact that there is no possible return to section $\Sigma$ when $a=D=0$, we assume that  $a^2+D^2\ne 0$ throughout this 
work. Note that this condition avoids the existence of a continuum of equilibria.

We are going to focus on the left Poincar\'{e} half-map (the one defined by the flow in the closed half-plane $\Sigma^-\cup\Sigma$ and the intersection points of its orbits with the Poincar\'e section $\Sigma$). Notice that the definition of the right Poincar\'{e} half-map and their corresponding results may be immediately obtained by the invariance of system \eqref{lienard} under the change $(x,y,a)\leftrightarrow (-x,-y,-a)$.

The left Poincar\'e half-map is usually defined in the following way.
Let us consider $(0,y_0)\in\Sigma$ with $y_0\geqslant0$ and let 
$
\Phi(t;y_0)=(\Phi_1(t;y_0),\Phi_2(t;y_0))
$
the solution of system \eqref{lienard} that satisfies the initial condition $\Phi(0;y_0)=(0,y_0)$. The existence of a value $\tau(y_0)>0$ such that $\Phi_1(\tau(y_0);y_0)=0$ and $\Phi_1(t;y_0)<0$ for every $t\in(0,\tau(y_0))$ allows to define the image of $y_0$ by the left Poincar\'e half-map as $P(y_0)=\Phi_2(\tau(y_0);y_0)\leqslant0$. Moreover, the value $\tau(y_0)$ is called the left flight time. 

Regarding the definition of the left Poincar\'e half-map at the origin, $P(0)$ cannot be defined as above when for every $\tau>0$ there exists $t\in(0,\tau)$ such that $\Phi_1(t;0)>0$. However, it can be continuously extended as $P(0)=0$ provided that for every $\varepsilon>0$ there exist $y_0\in(0,\varepsilon)$ and $y_1\in(-\varepsilon,0)$ such that $P(y_0)=y_1$. This finishes the usual definition of the left Poincar\'{e} half-map.

According to the above definition, it is natural to compute the flow of the system by means of explicit integrations. This leads to many case-by-case studies and forces the nonlinear appearance of the flight time. Here, we will use a characterization that avoids the computation of the flow, as it is done in  \cite{CARMONA2021319}. For the sake of completeness, we give a brief summary of the main results and ideas of \cite{CARMONA2021319} that are going to be used in this paper.

The left Poincar\'{e} half-map $P$ 
and its definition interval $I$ are given in Theorem 19 and Corollary 21 of \cite{CARMONA2021319}. By using the quadratic polynomial function
\begin{equation}\label{eq:poly}
W(y)=Dy^2-aTy+a^2,
\end{equation}
the {\it left Poincar\'{e} half-map}  is the unique function $P: I\subset [0,+\infty) \longrightarrow(-\infty,0]$
that, for every $y_0\in I$, satisfies 
\begin{equation}\label{integralF}
\operatorname{PV}\left\{\int_{P(y_0)}^{y_0}\dfrac{-y}{W(y)}dy\right\}= cT, 
\end{equation}
where $c$ is given, in terms of the parameters, as follows: (i) $c=0$ if $a>0$, (ii) $c=\pi\left(D\sqrt{4D-T^2}\right)^{-1}\in \mathbb{R}$ if $a=0,$ and (iii) $c=2\pi\left(D\sqrt{4D-T^2}\right)^{-1}\in \mathbb{R}$ if $a<0$. Here, $\operatorname{PV}\{\cdot\}$ stands for the {\it Cauchy Principal Value} at the origin  (see, for instance, \cite{henrici}), which is defined as
\[
\operatorname{PV}\left\{\int_{y_1}^{y_0}\dfrac{-y}{W(y)}dy\right\}=\lim_{\varepsilon\searrow 0} \left(\int_{y_1}^{-\varepsilon}\dfrac{-y}{W(y)}dy+\int_{\varepsilon}^{y_0}\dfrac{-y}{W(y)}dy\right),
\]
 for $y_1<0<y_0.$

As emphasized in \cite{CARMONA2021319}, the interval $I$ is essentially related with the roots of the quadratic polynomial function $W$. In the next remark, we shall briefly comment some of those relationships and other interesting properties of $P$ which are proven in \cite{CARMONA2021319}. 

\begin{remark}
\label{remark:intervlas}

System \eqref{lienard}, under the assumed condition $a^2+D^2\neq0$, has invariant straight lines for several values of the parameters.  These straight lines are either the invariant eigenspaces of equilibria (saddles, degenerate nodes or non-degenerate nodes) or the straight line $y=Tx+a/T$ in the case $T\neq0$, $D=0$ (what implies $a\neq0$).
In those cases, every invariant straight line intersects the Poincar\'e section $\Sigma$ in a point $(0,\mu)$, where $\mu$ is a root of  the quadratic polynomial function $W$ given in \eqref{eq:poly}. Moreover, when $I\subset [0,+\infty)$  is bounded, then the right endpoint of $I$ is a real root of  $W$ and, in the same way, if $P(I)$ is bounded, then the left endpoint of $P(I)$  is also a real root of  $W$. In Fig.~\ref{fig:halfmaps} (a) and Fig.~\ref{fig:halfmaps}(b), we show two examples of bounded intervals $I$ and/or $P(I)$, corresponding respectively to saddle and non-degenerate node configurations.

The interval $I$ can be unbounded. For instance, if $4D-T^2>0$, then the equilibrium point of system \eqref{lienard} is a focus or a center, 
the intervals $I$ and $P(I)$ are unbounded,
and, obviously, $P(y_0)$  tends to $-\infty$ as $y_0\to +\infty$. In this case, the intervals are $I=[0,+\infty)$ and $P(I)=(-\infty,0]$,
except when the equilibrium is a focus (i.e. $T\ne 0$) and it is located in the left half-plane $\left\{ (x,y)\in\mathbb{R}^2: x<0 \right\}$ (i.e. $a<0$). In fact, when $T>0$, the interval $P(I)$ is reduced to $(-\infty,\hat y_1]$, where  $\hat y_1=P(0)$ (see Fig.~\ref{fig:halfmaps} (c)).  Analogously,  for $T<0$, $I=[\hat y_0, +\infty)$ with $\hat y_0=P^{-1}(0)$ (see Fig.~\ref{fig:halfmaps} (d)).

Finally, the polynomial function $W$ is strictly positive in each set $[P(y_0),0)\cup(0,y_0 ]$,  with $y_0\in I$. Besides that, since $W(0)=a^2$, then $W(0)>0$ for $a\ne 0$ and $W(0)=0$ for $a=0$.

\end{remark}
\begin{figure}
\[
\begin{array}{cc}
\includegraphics[width=0.45\linewidth]{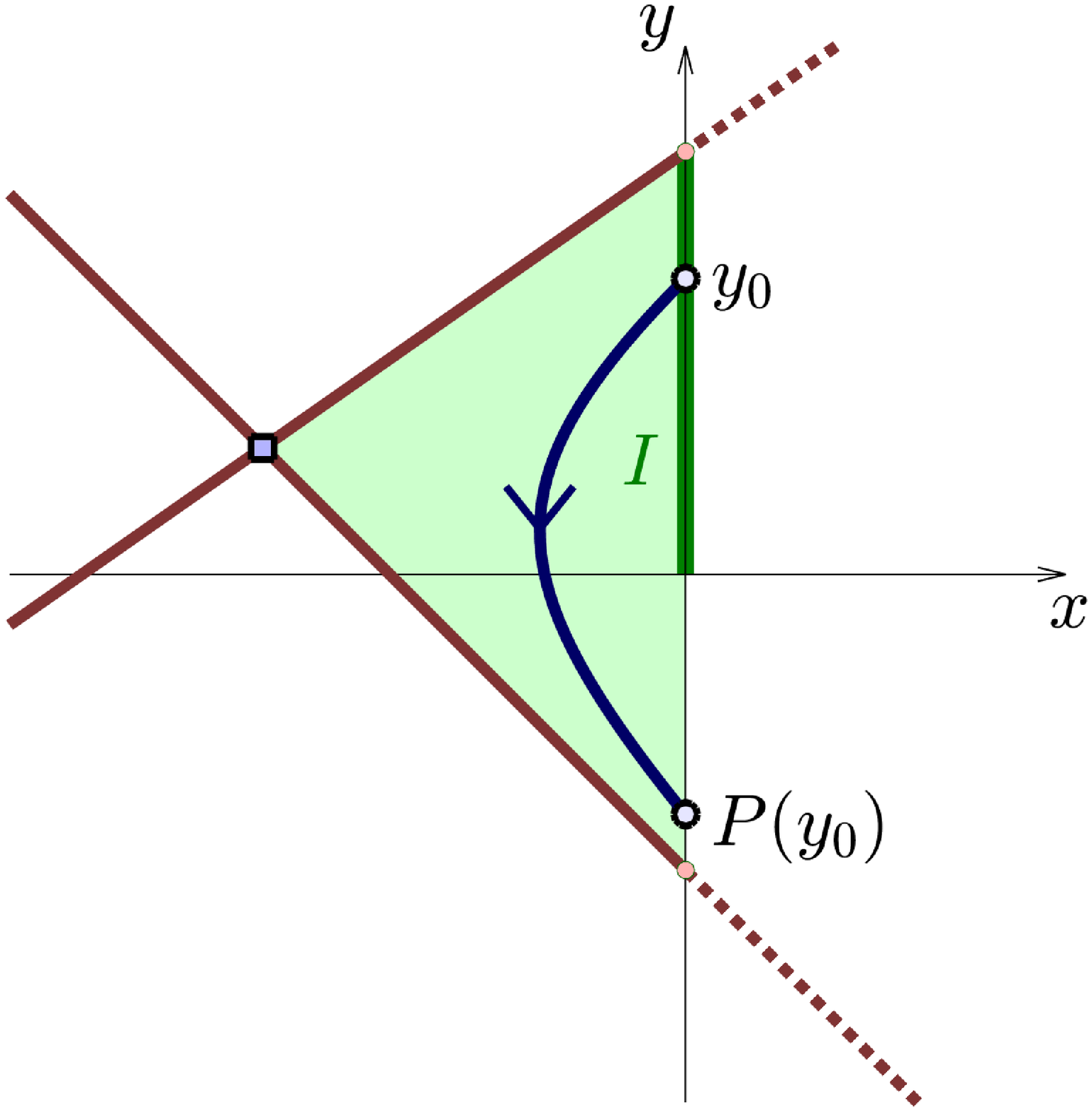}&
\includegraphics[width=0.45\linewidth]{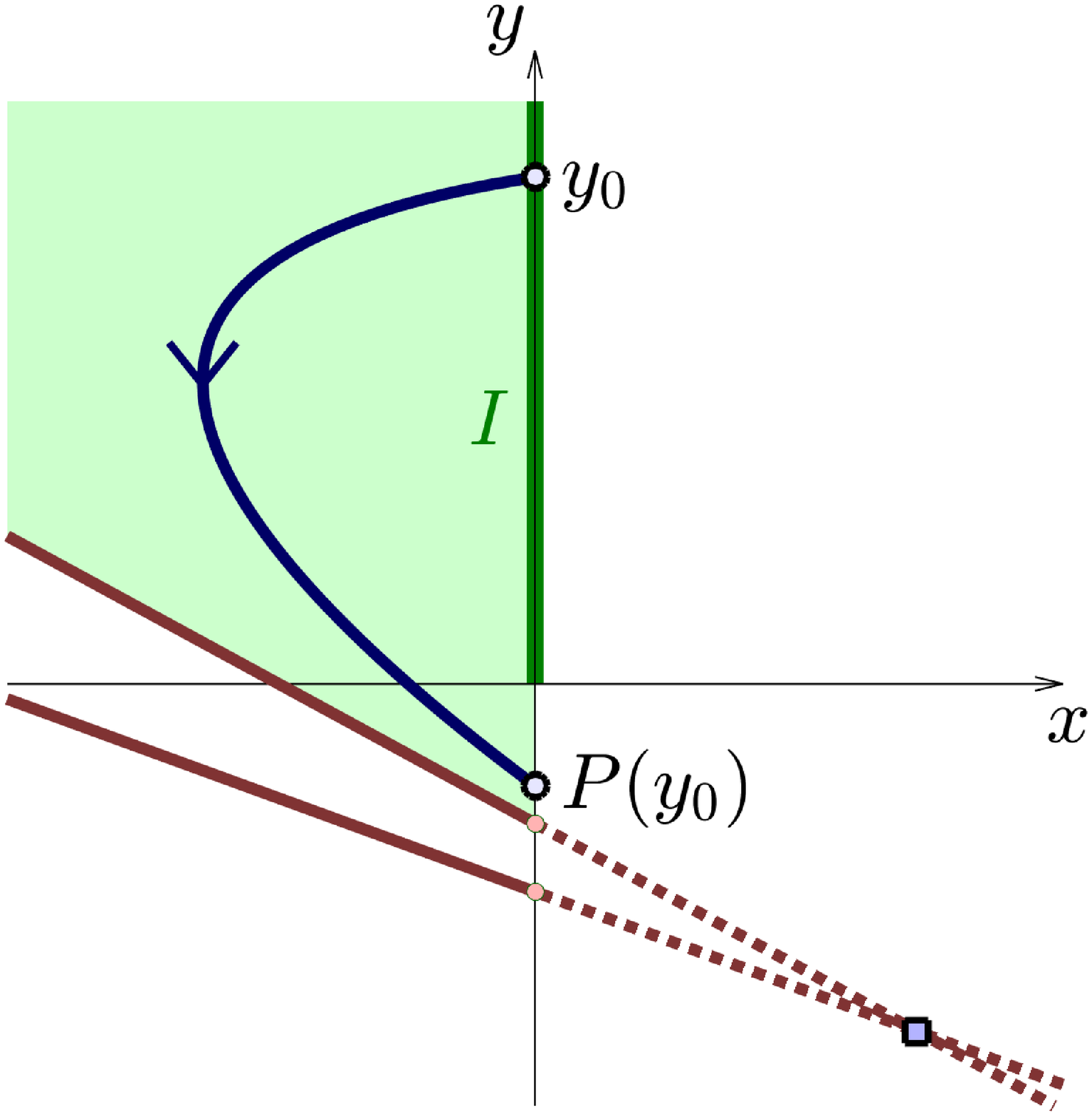}\\
\text{(a)}&\text{(b)}\\
\includegraphics[width=0.45\linewidth]{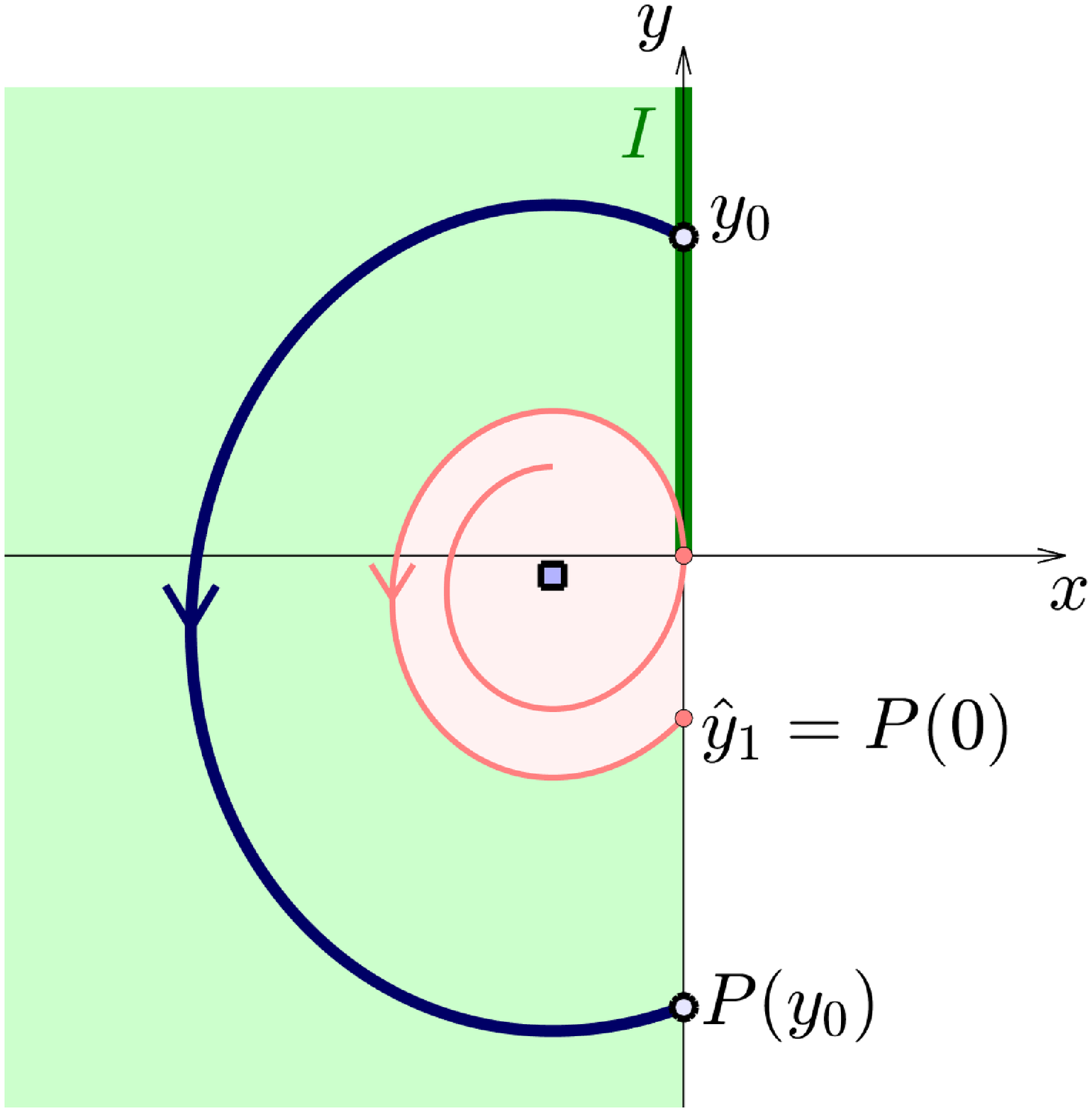}&
\includegraphics[width=0.45\linewidth]{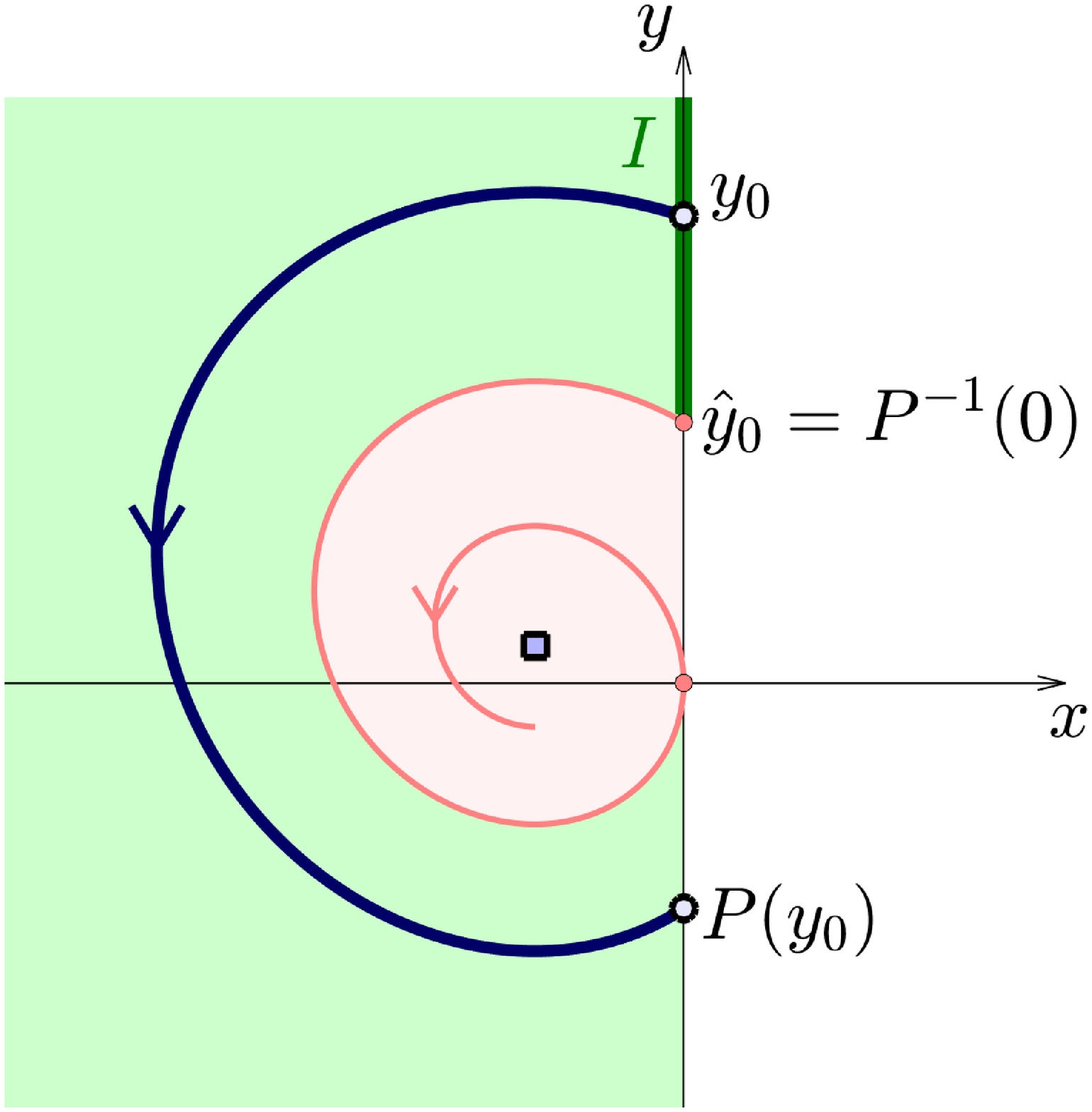}\\
\text{(c)}&\text{(d)}
\end{array}
\]
\caption{
The left Poincar\'e half-map $P$ and its interval of definition $I$ for the cases: (a) saddle, (b) non-degenerate node,
(c) unstable focus, and (d) stable focus. 
\label{fig:halfmaps}}
\end{figure}

It is worth mentioning that the integral given in \eqref{integralF} diverges when $a=0$ and the Cauchy principal value is necessary to overcome this difficulty. Moreover, in this case, for $y_1<0<y_0$, the Cauchy principal value at the origin is given by
\begin{equation}
\label{eq:PVvalue}
\operatorname{PV}\left\{\int_{y_1}^{y_0}\dfrac{-y}{Dy^2}dy\right\}=\lim_{\varepsilon\searrow 0} \left(\int_{y_1}^{-\varepsilon}\dfrac{-y}{Dy^2}dy+\int_{\varepsilon}^{y_0}\dfrac{-y}{Dy^2}dy\right)=\frac{1}{D}\log\left| \frac{y_1}{y_0}\right|.
\end{equation}
When $a\neq0,$ the integrating function $h(y)=-y/W(y)$ is continuous and, consequently, the Cauchy principal value just takes the value of the integral.

\section{Analiticity and series expansions of Poincar\'{e} half-maps at the tangency point and its preimage, and at infinity}\label{ssec:orin-inf}

In this section, by means of the integral characterization and a subsequent  differential equation, we shall compute the first coefficients of the Taylor expansion of the left Poincar\'{e} half-map $P$. Obviously, the used method does not depend on the spectrum of the matrix of the system. Before obtaining these coefficients it is necessary to determine the analyticity of the left Poincar\'{e} half-map $P$.  

When $P(y_0)\ne 0$, it is well-known (see, for example, \cite{Chicone82}) that the transversality between the flow of the system and the separation line $\Sigma$ ensures the analyticity of $P$ at $y_0$. The analyticity for the  tangency point between the flow and $\Sigma$ (that is, the origin) is more intricate and, in the literature, it has been approached  with a case-by-case study (see some partial results at \cite{COLL2001671}). However, as follows from Corollary 24 of \cite{CARMONA2021319}, the maps $P$ and $P^{-1}$ are real analytic functions in the open intervals $\operatorname{int}(I)$ and $P(\operatorname{int}(I))$, respectively, and at least one of the following statements is true: 
\begin{enumerate}
\item[(i)] the map $P$ is a real analytic function at the left endpoint of its domain, 
\item[(ii)] the map $P^{-1}$ is a real analytic function at the right endpoint of its domain.
\end{enumerate}

When the equilibrium of system \eqref{lienard} is a center or a focus, the left Poincar\'{e} half-map $P$ can be considered also at infinity. In addition, we shall obtain the first coefficients of the Taylor expansion of $P$  around the infinity. 

A first consequence from the definition of the left Poincar\'{e} half-map given in the integral form \eqref{integralF} is easily deduced by computing the derivative with respect to variable $y_0$ (see Remark 16 of \cite{CARMONA2021319}). Hence, one can see that the graph of the left Poincar\'{e} half-map $P$ and its inverse function $P^{-1}$, oriented according to increasing $y_0$, are particular orbits of the 
cubic vector field
\begin{equation*}\label{dy1}
\begin{array}{lcl}
X(y_0,y_1)&=&-\big(y_1W(y_0) ,y_0 W(y_1)\big)= \\
 \noalign{\medskip}
& & -\big(y_1\big(Dy_0^2-aTy_0+a^2\big) ,y_0 \big(Dy_1^2-aTy_1+a^2 \big)\big).
\end{array}
\end{equation*}
In fact,  the left Poincar\'{e} half-map $P$ and its inverse function $P^{-1}$ are solutions of the differential equation \begin{equation}
\label{eq:ode}
y_1W(y_0)dy_1-y_0W(y_1)dy_0=0.
\end{equation}

\vspace*{0.35cm}

The next proposition is a direct consequence of the results in \cite{CARMONA2021319} and allows to obtain the Taylor expansion of $P$ around the origin when $a\ne 0$ and $P(0)=0$. Notice that for $a=0$, the existence of  the left Poincar\'e half-map $P$ implies $4D-T^2>0$. From Remark  \ref{remark:intervlas}, the interval of definition of $P$ is $I=[0,+\infty)$ and, for $y_0\geqslant 0$, expression \eqref{integralF} can be written as 
\[
\operatorname{PV}\left\{\int_{P(y_0)}^{y_0}\dfrac{-y}{Dy^2}dy\right\}= \frac{\pi T}{D\sqrt{4D-T^2}}.
\]
Hence, by using the value for $\operatorname{PV}$ given in \eqref{eq:PVvalue}, the left Poincar\'{e} half-map $P$ is given by
\begin{equation}
\label{eq:y-1Ffora=0}
P(y_0)=-\exp\left(\frac{\pi T}{\sqrt{4D-T^2}} \right)y_0,\quad \mbox{for} \,\, y_0\geqslant 0.
\end{equation}

When $a\ne0$, by denoting $\mathcal{I}=\left\{ y\in\mathbb{R}: W(y)>0\right\}$, from Theorem 14 of \cite{CARMONA2021319}, it is deduced that  the set
\[
\mathcal{C}_0=\left\{ (y_1,y_0)\in\mathcal{I}^2: \int_{y_1}^{y_0}-y/W(y)dy=0\right\}
\]
can be written in the form
\[
\mathcal{C}_0=\left\{ (y_1,y_0)\in\mathcal{I}^2: (y_1-y_0)(y_1-\varphi_0(y_0)\right\},
\]
where $\varphi_0$ is a real analytic function in $\mathcal{I}$ which is also an involution, that is, $(\varphi_0 (\varphi_0(y_0)))=y_0$ 
for all $y_0\in \mathcal{I}$. Now, by means of Corollary 21 of  \cite{CARMONA2021319}, it follows that the left Poincar\'{e} half-map $P$ coincides with the function $\varphi_0$ restricted to the interval $\mathcal{I}\cap [0,\infty]$, provided $0\in I$ and $P(0)=0$. By an abuse of notation, we say that the Poincar\'{e} half-map $P$ is an involution when $a\ne0$, $0\in I$, and $P(0)=0$.

\begin{proposition}\label{prop:sequ}
Assume that $a\neq0$ and $0\in I$. If $P(0)=0$, then left Poincar\'{e} half-map $P$ is a real analytic function in $I$, it is an involution and its Taylor expansion around the origin writes as
\begin{equation}
\label{eq:maclaurin}
\begin{array}{ccl}
P(y_0)& =& {\displaystyle -y_0-\frac{2 T y_0^2}{3 a}-\frac{4 T^2 y_0^3}{9
   a^2}+\frac{2 \left(9 D T-22 T^3\right) y_0^4}{135
   a^3}+} \\
   \noalign{\medskip}
 & &    
   {\displaystyle\frac{4 \left(27 D T^2-26 T^4\right) y_0^5}{405
   a^4}-\frac{2 \left(27 D^2 T-176 D T^3+100 T^5\right)
   y_0^6}{945 a^5}+\mathcal{O} \left(y_0^7\right). }
  \end{array} 
\end{equation}

\end{proposition}
\begin{proof}
From the hypotheses of the proposition and by means of Theorem 14 and Corollary 24 of \cite{CARMONA2021319}, it is deduced that left Poincar\'{e} half-map $P$ 
is a real analytic function in $I$ and it is an involution. Hence, the derivative of $P$ at the origin is $P'(0)=-1$.

Now, taking into account that $P$ is a solution of the differential equation given in \eqref{eq:ode}, it is easy to obtain, via undetermined coefficients, the Taylor expansion given in \eqref{eq:maclaurin} and so the proof is concluded. 
\end{proof}
Notice that the Taylor expansion around the origin given in Proposition \ref{prop:sequ} was already obtained in \cite{MedradoTorregrosa15}. Although the calculations are not fully detailed in that work, the authors rely on the results given in  \cite{COLL2001671}, where the study requires different techniques depending on the situations. Before  \cite{MedradoTorregrosa15},  the same series expansion was obtained in \cite{FreireEtAl12}, by means of an inversion of the flight time, but only for the focus case.

When $0\in I$ and $P(0)\ne 0$, the function $P$ is a real analytic function at the origin and it is possible to obtain its  Taylor expansion of $P$ around the origin.
\begin{proposition}\label{prop:serietaylor}
Assume that 
$0\in I$. If $P(0)=\hat{y}_1<0$, then $a<0,T>0,$ $4D-T^2>0$, $\hat{y}_1$ is the right endpoint of the interval $P(I)$, and the left Poincar\'{e} half-map $P$ is a real analytic function in $I$ and its Taylor expansion around the origin writes as
\begin{equation}
\label{eq:serietaylor}
\begin{array}{r@{}c@{}l}
P(y_0)&=&{\displaystyle \hat{y}_1+\frac{W\left(\hat{y}_1\right) y_0^2}{2 a^2 \hat{y}_1}+\frac{T W\left(\hat{y}_1\right)
   y_0^3}{3 a^3 \hat{y}_1}-\frac{\left(a^2+\left(D-2 T^2\right)
   \hat{y}_1^2\right) W\left(\hat{y}_1\right) y_0^4}{8 a^4
   \hat{y}_1^3}} \\
   \noalign{\medskip}
   & - & 
   {\displaystyle 
   \frac{T \left(5 a^2+\left(7 D-6
   T^2\right) \hat{y}_1^2 W\left(\hat{y}_1\right)\right) y_0^5}{30 a^5
   \hat{y}_1^3}} \\
   \noalign{\medskip}
   &+ & 
   {\displaystyle
   \frac{\left(9 a^4-6 a^3 T \hat{y}_1+2 a^2 \left(9
   D-13 T^2\right) \hat{y}_1^2+\left(9 D^2-46 D T^2+24
   T^4\right) \hat{y}_1^4\right) W\left(\hat{y}_1\right) y_0^6}{144 a^6
   \hat{y}_1^5}} \\
    \noalign{\medskip}
   &+ &  
   \mathcal{O}\left(y_0^7\right).
   \end{array}
\end{equation}
\end{proposition}
\begin{proof}
The expression given in \eqref{eq:y-1Ffora=0} provides the left Poincar\'e half-map for the case $a=0$. From there, one obtains $P(0)=0$ when $a=0$. 
 
Suppose that $0\in I$ and $P(0)=\hat{y}_1<0$. Then $a\ne 0$ and expression \eqref{integralF} leads us to 
\[
\int_{\hat{y}_1}^{0}\dfrac{-y}{W(y)}dy= cT.
\]
From Remark \ref{remark:intervlas}, the polynomial $W$ is strictly positive and, therefore, the left-hand term of the last expression is also strictly positive. If $a>0$, from expression \eqref{integralF}, $c=0$ and this is impossible. Thus, it is deduced that $a<0,T>0,$ and $4D-T^2>0$. Now, from Remark \ref{remark:intervlas} again, the intervals $I$ and $P(I)$ are unbounded and $P(y_0)$  tends to $-\infty$ as $y_0\to +\infty$.

Next, let us  prove that $\hat{y}_1$ is the  right endpoint of the interval $I$.  Let us consider $y_0\geqslant 0$ and $y_1\in (\hat{y}_1,0]$. From the inequalities 
\[
\int_{y_1}^{y_0}\dfrac{-y}{W(y)}dy<\int_{\hat{y}_1}^{y_0}\dfrac{-y}{W(y)}dy\leqslant \int_{\hat{y}_1}^{0}\dfrac{-y}{W(y)}dy=cT
\]
one can see that no point in the interval $y_1\in (\hat{y}_1,0]$ belongs to the  interval $P(I)$ and so the right endpoint of interval $P(I)$ is $\hat{y}_1$. 

The analyticity of $P$ is a direct consequence of Theorem 14 and Corollary 21 of \cite{CARMONA2021319} and the Taylor expansion around the origin given in \eqref{eq:serietaylor} follows from the method of undetermined coefficients applied to the differential equation \eqref{eq:ode}.
\end{proof}
\begin{remark}
\label{remark:focus1}
Note that the condition $P(0)=\hat{y}_1<0$ together with the linearity of the system implies that there exists an unstable focus equilibrium in the left half-plane $\Sigma^-$ (see Fig. \ref{fig:halfmaps}(c)) and so it is immediate that 
$a<0,T>0,$ and $4D-T^2>0$. This is an alternative proof for  the inequalities of Proposition \ref{prop:serietaylor}. On the other hand, the endpoints of intervals $I$ and $P(I)$ were also determined in Corollary 21 of \cite{CARMONA2021319} in a more generic way. For the sake of completeness, in the previous proof, we have included a different and specific reasoning for this case. 
\end{remark}

When there exists a point  $\hat{y}_0>0$ such that $P\left(\hat{y}_0\right)=0$, then left Poincar\'{e} half-map $P$ is a non-analytic function at $\hat{y}_0$. However, in \cite{CARMONA2021319} it is proven that the inverse function $P^{-1}$ is analytic at the origin and so it is possible, by means of an inversion, to get a Newton-Puiseux serie expansion for the left Poincar\'{e} half-map $P$ around $\hat{y}_0$. Some results about series inversion and Newton-Puiseux series can be found in 
\cite{Garcia-Barroso:2017ty} and the references therein. Also of interest are the results included in \cite{CANO2020} concerning the expression of the solutions of differential equations as Newton-Puiseux series expansion and its convergence.

\begin{proposition}\label{prop:serieNewton-Puiseux}
Assume that there exists a value $\hat{y}_0>0$ with $P\left(\hat{y}_0\right)=0$. Then, 
$a<0,T<0,$ $4D-T^2>0$, $\hat{y}_0$ is the left endpoint of the interval $I$, the inverse function $P^{-1}$ is a real analytic function, and the left Poincar\'{e} half-map $P$ admits the Newton-Puiseux serie expansion around the point $\hat{y}_0$ given by
\begin{equation}
\label{eq:serieNewton-Puiseux}
\begin{array}{ccl}
P(y_0)&=& {\displaystyle a\sqrt{\frac{2\hat{y}_0}{W(\hat{y}_0)}}\,(y_0-\hat{y}_0)^{1/2}-\frac{aT}{3}\frac{2\hat{y}_0}{W(\hat{y}_0)}(y_0-\hat{y}_0)+} \\
\noalign{\medskip} 
& & 
{\displaystyle 
\frac{a^3}{72}\left( \frac{9D+2T^2}{a^2}+
\frac{9}{\hat{y}_0^2}\right) \left(\sqrt{\frac{2\hat{y}_0}{W(\hat{y}_0)}}\right)^3 (y_0-\hat{y}_0)^{3/2}+\mathcal{O}\left((y_0-\hat{y}_0)^2\right),}
\end{array}
\end{equation}
which is valid for $y_0\geqslant \hat{y}_0$.
\end{proposition}
\begin{proof}
Suppose that there exists a point $\hat{y}_0>0$ such that $P(\hat{y}_0)=0$.
An analogous reasoning to the first part of the proof of Proposition \ref{prop:serietaylor} leads to the inequalities $a<0,T<0,$ and $4D-T^2>0$ and to the fact that $\hat{y}_0$ is the left endpoint of the interval $I$.

The inverse function $P^{-1}$ satisfies $P^{-1}(0)=\hat{y}_0$ and, from differential equation \eqref{eq:ode}, it follows that its derivative at the origin vanishes. This implies that $P$ is a non-analytic function at  $\hat{y}_0$. From Theorem 14 and Corollary 21 of \cite{CARMONA2021319}, it follows that the inverse function $P^{-1}$ is an analytic function at the origin and $P^{-1}$ admits the Taylor expansion \eqref{eq:serietaylor} by changing $\hat{y}_1$ by $\hat{y}_0$.

Now, the Newton-Puiseux series expansion of $P$ is obtained by the inversion of the Taylor expansion of $P^{-1}$. Note that the direct inversion provides two possible series expansions but, since $P(y_0)\leqslant 0$ for all $y_0\in I$, the valid one is that given in \eqref{eq:serieNewton-Puiseux} and the proof is finished. 
\end{proof}

An analogous comment  to Remark \ref{remark:focus1} can be made about the inequalities of Proposition \ref{prop:serieNewton-Puiseux} and the left endpoint of $I$. The scenario described by the hypothesis stated in Proposition \ref{prop:serieNewton-Puiseux} is illustrated in Fig. \ref{fig:halfmaps}(d).

\begin{remark}
\label{remark:inversion}
The inversion used to obtain the Newton-Puiseux series expansion of $P$ is equivalent to the computation of the Taylor expansion of $Q(z_0):=P(\hat{y}_0+z_0^2)$ around $z_0=0$ and the subsequent change of $z_0$ by $\sqrt{y_0-\hat{y}_0}$. In order to get this Taylor expansion it is enough to  make the change of variable $y_0\to \hat{y}_0+z_0^2$ in the differential equation \eqref{eq:ode} to achieve a differential equation for the function $Q$.  
\end{remark}

Let us recall from Remark \ref{remark:intervlas} that when $4D-T^2>0$ the domain $I$ 
is unbounded with $P(y_0)$ 
tending to $-\infty$ as $y_0\to +\infty$. Thus, the study of the left Poincar\'{e} half-map around the infinity  is feasible. In fact, the first two terms of  the Taylor expansion of left Poincar\'{e}  half-map $P$ around the infinity were already obtained in \cite{FreireEtAl98} by means of an expression of $P$ parameterized by the flight time. In the following proposition, we present a simple method to get these and others terms. 
\begin{proposition}\label{prop:taylor-inf}
Assume that $4D-T^2>0$.Then, the Taylor expansion of left Poincar\'{e}  half-map $P$ 
around the infinity writes as
\begin{equation*}
\label{taylor-inf}
\begin{array}{ccl}
P(y_0) & = & {\displaystyle-\exp\left(\frac{\pi T}{\sqrt{4D-T^2}} \right)y_0+\frac{aT}{D}\left( 1+ \exp\left(\frac{\pi T}{\sqrt{4D-T^2}} \right) \right)-} \\
\noalign{\medskip}
& & 
{\displaystyle \frac{a^2}{D}\sinh\left(\frac{\pi T}{\sqrt{4D-T^2}} \right)\cdot\frac{1}{y_0} - }\\
\noalign{\medskip}
& &{\displaystyle \frac{a^3 e^{-\frac{2 \pi  T}{\sqrt{4 D-T^2}}} \left(-2+e^{\frac{\pi 
   T}{\sqrt{4 D-T^2}}}\right) \left(1+e^{\frac{\pi  T}{\sqrt{4
   D-T^2}}}\right)^2 T}{6 D^2}\cdot\frac{1}{y_0^2}+\mathcal{O}\left( \frac{1}{y_0^3}\right).}
\end{array}
\end{equation*}
\end{proposition}

\begin{proof} Firstly, we shall prove the equality 
\begin{equation}
\label{eq:taylor-y1Finf}
\lim_{y_0\to +\infty}\frac{P(y_0)}{y_0}=-\exp\left(\frac{\pi T}{\sqrt{4D-T^2}}\right).
\end{equation}

If $a=0$, then expression  \eqref{eq:y-1Ffora=0} leads us directly to equality \eqref{eq:taylor-y1Finf}.

If $a\ne0$, taking into account that $W(y)>0$ for $y\in\mathbb{R}$ (see Remark \ref{remark:intervlas}), then relationship \eqref{integralF} can be written as
\begin{equation}
\label{eq:splitintegral2}
\int_{P(y_0)}^{-y_0}\dfrac{-y}{W(y)}dy+\int_{-y_0}^{y_0}\dfrac{-y}{W(y)}dy=cT,
\end{equation}
for  $y_0\in I$, being $c=0$ for $a>0$ and $c=2\pi\left(D\sqrt{4D-T^2}\right)^{-1}$ for $a<0$.

The change of variable $Y=1/y$ applied to the first integral in expression \eqref{eq:splitintegral2} transforms it into 
\[
\int_{1/P(y_0)}^{-1/y_0}\dfrac{1}{a^2Y^3-aTY^2+DY}dY+\int_{-y_0}^{y_0}\dfrac{-y}{W(y)}dy= cT 
\]
or, equivalently, into the expression
\[
\int_{1/P(y_0)}^{-1/y_0}\frac{1}{DY} dY+ \int_{1/P(y_0)}^{-1/y_0}\frac{a(T-aY)}{D\left(a^2Y^2-aTY+D \right)}dY+
\int_{-y_0}^{y_0}\dfrac{-y}{W(y)}dy=cT.
\]
That is,
\[
 \frac{P(y_0)}{y_0}= -\exp\left(DcT - D\int_{-y_0}^{y_0}\frac{-y}{W(y)}dy-
\int_{1/P(y_0)}^{-1/y_0}\frac{a(T-aY)}{a^2Y^2-aTY+D}dY\right).
\]

Now, a direct integration provides
\begin{equation}
\label{eq:PVinfinito}
\lim_{y_0\to+\infty}\int_{-y_0}^{y_0}\dfrac{-y}{W(y)}dy=-\dfrac{\pi T\sgn(a)}{D\sqrt{4D-T^2}}
\end{equation}
and taking into account that 
\[
\lim_{y_0\to +\infty} \int_{1/P(y_0)}^{-1/y_0}\frac{a(T-aY)}{a^2Y^2-aTY+D}dY=0,
\]
the equality \eqref{eq:taylor-y1Finf} follows.

Thus, the function $\tilde P,$ defined by 
\[
\tilde P(Y_0)=\left\{
\begin{array}{ccl}
{\displaystyle \frac{1}{P(1/Y_0)} }& \mbox{if}  & Y_0\ne 0 \,\, \mbox{and} \,\, 1/Y_0\in I, \\
\noalign{\medskip}
0 & \mbox{if}  & Y_0=0,
\end{array}
\right.
\]
has derivative on the right at the origin and its value is 
\begin{equation}
\label{eq:der1}
\alpha_1:=\frac{d\tilde P}{dY_0}\left(0^+\right)=-\exp\left(\frac{-\pi T}{\sqrt{4D-T^2}}\right).
\end{equation}
Moreover, it is immediate to see that the function $\tilde P$ is a solution of differential equation
\[
\left( a^2Y_0^2-aTY_0+D\right)Y_0dY_1-\left( a^2Y_1^2-aTY_1+D\right)Y_1dY_0=0,
\]
obtained from the differential equation \eqref{eq:ode} by means of the change of variables $(Y_0,Y_1)=(1/y_0,1/y_1)$ (defined for $y_0y_1\ne 0$). 
From here, it is deduced that the function $\tilde P$ has derivatives on the right of all orders at $Y_0=0$ and, after a direct computation, one finds
\begin{equation*}
\label{eq:der2}
\alpha_2:=\frac{d^2\tilde P}{dY_0^2}\left(0^+\right)=
-\frac{2a T}{D} e^{-\frac{2 \pi  T}{\sqrt{4D-T^2}}} \left(e^{\frac{\pi  T}{\sqrt{4
   D-T^2}}}+1\right),
\end{equation*}
\begin{equation*}
\label{eq:der3}
\begin{array}{l}
\alpha_3:={\displaystyle \frac{d^3\tilde P}{dY_0^3}\left(0^+\right)=} \\
\noalign{\medskip}
\qquad {\displaystyle \frac{3 a^2}{D^2} e^{-\frac{3 \pi  T}{\sqrt{4D-T^2}}} \left(e^{\frac{\pi  T}{\sqrt{4
   D-T^2}}}+1\right) \left(-2 T^2 e^{\frac{\pi  T}{\sqrt{4
   D-T^2}}}+De^{\frac{\pi  T}{\sqrt{4 D-T^2}}}-D-2
   T^2\right),}
   \end{array}
\end{equation*}
and
\begin{equation}
\label{eq:der4}
\begin{array}{l}
\alpha_4:={\displaystyle \frac{d^4\tilde P}{dY_0^4}\left(0^+\right)=} \\
\noalign{\medskip}
\qquad {\displaystyle \frac{4 a^3 T} {D^3}e^{-\frac{4 \pi  T}{\sqrt{4 D-T^2}}} \left(1+e^{\frac{\pi
    T}{\sqrt{4 D-T^2}}}\right)^2  \left(-8 D+7 D e^{\frac{\pi 
   T}{\sqrt{4 D-T^2}}}-6 T^2-6 T^2e^{\frac{\pi  T}{\sqrt{4 D-T^2}}}
   \right).}
   \end{array}
\end{equation}

Since the Taylor expansion of left Poincar\'e half-map $P$ around the infinity is given by
\begin{equation}
\label{eq:seriey-1F}
P(y_0)=\frac{1}{\alpha_1}y_0-\frac{\alpha_2}{2\alpha_1^2}+\frac{3\alpha_2^2-2\alpha_1\alpha_3}{12\alpha_1^3}\cdot\frac{1}{y_0}-\frac{3 \alpha _2^3-4 \alpha _1 \alpha _2 \alpha _3+\alpha _1^2
   \alpha _4}{24 \alpha _1^4} \cdot\frac{1}{y_0^2}+  \mathcal{O}\left(\frac{1}{y_0^3} \right),
\end{equation}
the proof concludes by substituting expressions \eqref{eq:der1}--\eqref{eq:der4} into \eqref{eq:seriey-1F}.
 \end{proof}

\section{The relative position between the graph of Poincar\'{e} half-maps and the bisector of the fourth quadrant}\label{ssec:relativeposition}

To study the relative position between the graph of the left  Poincar\'{e} half-map and the bisector of the fourth quadrant, it is natural to analyze the sign of the difference $y_0-(-P(y_0))$. In the next proposition, we show the relationship between this difference and the trace $T$. Notice that this relationship has been addressed via a case-by-case treatment (by distinguishing the spectrum of the matrix of the system) in the main results of chapter 4 of \cite{LT14}. Here, we provide a concise proof by using the integral characterization of the left Poincar\'{e} half-map.

\begin{proposition}
\label{rm:signoy0+y1}
The left Poincar\'{e} half-map $P$ satisfies the relationship
\[
\sgn\left(y_0+P(y_0) \right)=-\sgn(T) \quad \mbox{for} \quad y_0\in  I\setminus\{0\}.
\]
In addition, 
when $0\in I$ and $P(0)\neq0$ or when $T=0$, the relationship also holds for $y_0=0$.
\end{proposition}
\begin{proof} 
We will prove this proposition by distinguishing the cases $T=0$ and $T\ne 0$.

For $T=0$, the integral equation given in \eqref{integralF} is reduced to
\[
PV\left\{\int_{P(y_0)}^{y_0}\dfrac{-y}{Dy^2+a^2}dy\right\}= 0, \quad \mbox{for} \quad y_0\in I.
\]
By taking into account that the integrating function is an odd function, it is direct to see that $P(y_0)=-y_0$ for  all $ y_0\in I$  and so the proposition is true for $T=0$.

Now, we focus on the proof for the case $T\neq0$ and we will consider the situations $a=0$ and $a\ne 0$. 

When $a=0$, the left Poincar\'e half-map $P$ is given by expression \eqref{eq:y-1Ffora=0} and so the equality $\sgn\left(y_0+P(y_0) \right)=-\sgn(T)$ holds for every $y_0\in I.$

When $a\ne 0$, let us consider the interval \[J=\left\{ u\in\mathbb{R}:  W(y)>0, \forall\,y\in[-|u|,|u|]\right\}\] and function $g:J\longrightarrow \mathbb{R}$ defined by
\[
g(u)=\int_{-u}^{u}\dfrac{-y}{W(y)}dy,
\]
where $W$ is the polynomial function defined in \eqref{eq:poly}. 

Notice that function g satisfies $g(0)=0$, its derivative is
\[
g'(u)=\dfrac{-2 a T u^2}{W(u)W(-u)}
\]
and so $\sgn(g'(u))=-\sgn(a T)$ for every $u\in J\setminus\{0\}$. Thus, $\sgn(g(u))=-\sgn(a T)$ for every $u\in J\cap (0,+\infty)$ and $\sgn(g(u))=\sgn(a T)$ for every $u\in J\cap (-\infty,0)$. 

Moreover, if $J=\mathbb{R} $ (i.e., when $4D-T^2>0$), then, from \eqref{eq:PVinfinito},
\[
\lim_{u\to+\infty}g(u)=-\dfrac{\pi T\sgn(a)}{D\sqrt{4D-T^2}}.
\]

The existence of the left Poincar\'{e} half-map $P$ for the case $a\ne 0$ implies $a>0$ and $c=0$ or $a<0$ and $c=2\pi\left(D\sqrt{4D-T^2}\right)^{-1}\in \mathbb{R}$. It is straightforward to see that these conditions together with the  properties of function $g$ lead to the equality 
\begin{equation}\label{eq:aux}
\sgn(cT-g(u))=\sgn(T).
\end{equation}

Let us consider $y_0\in \operatorname{int}(I)\cap J$. From equality \eqref{integralF}, one gets 
\begin{equation*}\label{splitintegral}
cT=\int_{P(y_0)}^{y_0}\dfrac{-y}{W(y)}dy=\int_{P(y_0)}^{-y_0}\dfrac{-y}{W(y)}dy+\int_{-y_0}^{y_0}\dfrac{-y}{W(y)}dy,
\end{equation*}
that is,
\[
 \int_{P(y_0)}^{-y_0}\dfrac{-y}{W(y)}dy=cT-g(y_0).
 \]
Thus, from \eqref{eq:aux},
\begin{equation*}
\label{ecu:signyo+y1andT}
\sgn\left( \int_{P(y_0)}^{-y_0}\dfrac{-y}{W(y)}dy \right)=\sgn(T)\ne 0
\end{equation*}
and, taking into account that $-y_0\cdot P(y_0)\geqslant 0$, equality $\sgn\left(y_0+P(y_0) \right)=-\sgn(T)$ holds for every $y_0\in \operatorname{int}(I)\cap J$. Therefore, the conclusion follows by using the continuity of the left Poincar\'{e} half-map and the function $y_1(y_0)=-y_0$.
\end{proof}

The next result establishes, as a direct consequence of Proposition \ref{rm:signoy0+y1}, the relationship between the graph of the left  Poincar\'{e} half-map and the bisector of the fourth quadrant.

\begin{corollary}\label{rm:bi2quad}
The following items are true.
\begin{enumerate}
\item If $T=0$, then the graph of the left Poincar\'{e} half-map $P$ of system \eqref{lienard} associated to section $\Sigma\equiv \{x=0\}$, if it exists, is included in the bisector of the fourth quadrant. 
\item If $T> 0$ (resp. $T<0$), then the graph of left Poincar\'{e} half-map $P$ of system \eqref{lienard} associated to section $\Sigma\equiv \{x=0\}$, if it exists, is located below (resp. above) the bisector of the fourth quadrant except perhaps at the origin.
\end{enumerate}
\end{corollary}

\section{The sign of the second derivative of Poincar\'{e} half-maps}
\label{sec:derivadasegunda}
From the differential equation given in \eqref{eq:ode}, it is easy to obtain explicit expressions for the derivatives of $P$ with respect to $y_0$. The first and second derivatives are shown in the next result. Its proof is a simple computation and so it is omitted. 
\begin{proposition}
\label{prop:derivadasprimeraysegunda}
The first and second derivatives of the left Poincar\'{e} half-map $P$ with respect to $y_0$, in the interval $\operatorname{int}(I)$, are given by 
\begin{equation*}
\label{eq:derivadaprimera}
\frac{dP}{dy_0}(y_0)=\frac{y_0W(P(y_0))}{P(y_0)W(y_0)}\,
\end{equation*}
and
\begin{equation}
\label{eq:derivadasegunda}
\frac{d^2P}{dy_0^2}(y_0)=-\frac{a^2 \left(y_0^2-\left(P(y_0)\right)^2\right)W(P(y_0))}{\left(P(y_0)\right)^3\left(W(y_0)\right)^2}\,.
\end{equation}
\end{proposition}

As will be stated in the next section, some interesting applications to periodic behavior of piecewise linear systems come out from the signs of the first and the second derivatives of $P$. Note that the sign of the first derivative is obvious from \eqref{eq:derivadaprimera}, because $y_0P(y_0)<0$ for $y_0\in\operatorname{int}(I)$ and the polynomial $W$ is positive (see Remark \ref{remark:intervlas}). Besides that, the sign of the second derivative of left Poincar\'{e} half-map $P$ is an immediate consequence of expression given in \eqref{eq:derivadasegunda} and Proposition \ref{rm:signoy0+y1}.

\begin{proposition}\label{prop:signderivadasegunda}
The sign of the second derivative of left Poincar\'{e} half-map $P$ is given by
\[
\sgn\left(\frac{d^2P}{dy_0^2}(y_0) \right)=-\sgn(a^2T) \quad \mbox{for} \quad y_0\in  \operatorname{int}(I).
\]
\end{proposition}
Note that $a^2$ is written in the previous expression to include the case $a=0$.

 In previous works, the sign of the second derivative of the Poincar\'e half-maps has been addressed via case-by-case studies (see, for instance \cite{LT14}), where distinguished analyses must be employed for different values of the parameters.
Nevertheless, in Proposition \ref{prop:signderivadasegunda}, the integral characterization has allowed to obtain a closed expression for such a sign regardless the cases. As far as we know, this common expression has not been previously obtained in the literature.

\section{Some immediate consequences in piecewise linear systems}
\label{sec:consequences}
The previous results established some fundamental properties of Poincar\'e half-maps defined on a straight line for planar linear systems. These properties are essential to understand the dynamic behavior of planar piecewise linear systems. This section is devoted to provide some immediate consequences regarding periodic behavior in piecewise linear systems with two zones separated by a straight line.

From Freire et. al in \cite[Proposition 3.1]{FreireEtAl12}, we known that any piecewise linear system with two zones separated by a straight line $\Sigma$ for which a Poincar\'{e} map is well defined can be written in the  following Li\'enard canonical form 
\begin{equation}\label{cf}
\left\{\begin{array}{l}
\dot x= T_L x-y\\
\dot y= D_L x-a_L
\end{array}\right.\quad \text{for}\quad x< 0,
\quad 
\left\{\begin{array}{l}
\dot x= T_R x-y+b\\
\dot y= D_R x-a_R
\end{array}\right.\quad \text{for}\quad x> 0.
\end{equation}

Note that the points $(0,0)$ and $(0,b)$ are the tangency points between $\Sigma$ and, respectively, the flow of the left and right systems.
When $b=0$ these points coincide and the flow of system \eqref{cf} crosses the separation line transversally except at the origin. In this case, the system is called {\it sewing}.
On the contrary, for $b\ne0$  the flow of system \eqref{cf} does not cross the separation line along the segment
\[
\Sigma_s=\left\{ (0,\mu+(1-\mu)b)\in\Sigma: \mu\in(0,1)\right\},
\] 
which is usually called the {\it sliding region}.

In order to analyze the behaviour of system \eqref{cf} we consider two Poincar\'{e} half-maps associated to $\Sigma$, to wit, the {\it Forward Poincar\'{e} half-map}  $y_L: I_L\subset [0,+\infty) \longrightarrow(-\infty,0]$ and  the {\it Backward Poincar\'{e} half-map} $y_R:I_R\subset [b,+\infty)\rightarrow (-\infty,b]$. The forward one goes in the positive direction of the flow and maps a point $(0,y_0)$, with $y_0\geqslant 0$, to a point $(0,y_L(y_0))$. Analogously, the backward one goes in the negative direction of the flow and maps a point $(0,y_0)$, with $y_0\geqslant b$, to $(0,y_R(y_0))$. Notice that $y_L$ is defined by the left system and $y_R$ is defined by the right system. Naturally, $y_L=P$ by taking $T=T_L, D=D_L,$ and $a=a_L$ in system \eqref{lienard}. In addition, taking into account the change $(t,x)\to -(t,x)$, one has $y_R(y_0)=P(y_0-b)+b$ by taking $T=-T_R, D=D_R$, and $a=-a_R$. 

Evidently, the intersections between the curves $y_1=y_L(y_0)$ and $y_1=y_R(y_0),$ for $y_0\in\Int(I_L\cap I_R),$ are in bijective correspondence to crossing periodic solutions of \eqref{cf}.

From Proposition \ref{rm:signoy0+y1},
\begin{equation}\label{eq:prop5right}
T_L=0 \,(\text{resp. } T_R=0)\Rightarrow y_L(y_0)=-y_0\, (\text{resp. } y_R(y_0)=-y_0+2b),
\end{equation}
when, of course, the map $y_L$ (resp. $y_R$) exists.  Therefore, the following result follows immediately. 

\begin{corollary}\label{cor:continuum}
Assume that $T_L^2+T_R^2=0$. If $b\neq 0$, the system \eqref{cf} does not have crossing periodic orbits. If $b=0$ and $\Int(I_L\cap I_R)\neq \emptyset$, then it has a continuum of crossing periodic orbits.
\end{corollary}

It is also possible to give some results for the case $T_L^2+T_R^2>0$. From Corollary \ref{rm:bi2quad}, if $T_L>0$ (resp. $T_L<0$), then the curve $y_1=y_L(y_0)$, if it exists, is located below (resp. above) the straight line $y_1=-y_0$ except perhaps at the origin. Analogously, if $T_R>0$ (resp. $T_R<0$), the curve $y_1=y_R(y_0)$, if it exists, is located above (resp. below) the straight line $y_1=-y_0+2b$ except perhaps at the point $(b,b)$. Hence,  also taking \eqref{eq:prop5right} into account, if $T_L>0$, $T_R\geqslant0$, and $b\geqslant0$, then
\[
y_L(y_0)<-y_0\leqslant-y_0+2b\leq y_R(y_0).
\]
Therefore, the graphs of $y_L$ and $y_R$ have no intersection points and so system \eqref{cf} has no crossing periodic orbits. The following result about non-existence of periodic orbits follows immediately via a similar reasoning.

\begin{corollary}\label{cor:nonexistence}
\label{cor:TLTR>=0}
Assume that $T_LT_R\geqslant0$
and that one of the following two non-exclusive hypotheses holds: 
\begin{itemize}
\item[1)]$T_L\neq0$ and $T_L b\geqslant0$;
\item[2)] $T_R\neq0$ and $T_R b\geqslant0$. 
\end{itemize}
Then, system \eqref{cf} does not have crossing periodic orbits.
\end{corollary}

By merging the information of Corollaries \ref{cor:continuum} and \ref{cor:nonexistence}, we get that if $T_LT_R\geqslant0$ and $T_L b\geqslant0$, then system \eqref{cf} either does not have crossing periodic orbits or has a continuum of crossing periodic orbits. In other words, it does not  have limit cycles.

Let us add some lines regarding the condition $T_LT_R\geqslant0$ added in Corollary \ref{cor:TLTR>=0}. Note that for the case $b=0$ it is well known  that $T_LT_R\leqslant0$ is a necessary condition for the existence of crossing periodic orbits (see, for instance, \cite{FreireEtAl12}). However, when $b\neq 0$ they could exist even for $T_LT_R>0$.  Thus, the previous result allows to remove some cases where crossing periodic solutions do not exist.

The obtention of the previous results relies only on the relative location of the graphs of the Poincar\'e half-maps, which 
is easily determined in terms of the basic parameters $T_L$, $T_R$, and $b$, by using Proposition \ref{rm:signoy0+y1}. Now, information about the shape of their graphs, revealed by Proposition \ref{prop:signderivadasegunda}, can be used to bound the number of limit cycles of piecewise linear systems (isolated crossing periodic solutions) in some generic cases.
In fact, from Proposition \ref{prop:signderivadasegunda}, a simple expression is obtained for the sign of the second derivatives
 of the Poincar\'e half-maps,
\[
\sgn\left(\frac{d^2y_L}{dy_0^2}(y_0) \right)=-\sgn(a_L^2T_L) \quad \text{and}\quad\sgn\left(\frac{d^2y_R}{dy_0^2}(y_0) \right)=\sgn(a_R^2T_R).
\]
Therefore, the concavity of the functions $y_L$ and $y_R$ is established by $a_L^2T_L$ and $a_R^2T_R$, respectively. Thus, the following result for limit cycles  follows immediately.

\begin{corollary}\label{cor:atmost2}
If $T_L\,T_R>0$, then system \eqref{cf} has at most two limit cycles.
\end{corollary}

The upper bound given by the corollary above is reachable. Indeed, the last example provided by Han and Zhang in \cite{HanZhang10} satisfies $T_L\,T_R>0$ and has two limit cycles near the origin.

Concerning the series expansions of Poincar\'{e} half-maps provided by Propositions \ref{prop:sequ}-\ref{prop:taylor-inf}, a natural application could consist in obtaining stability properties of some singular invariant sets of piecewise linear systems under suitable assumptions. For instance, Proposition \ref{prop:sequ} can provide whether the monodromic singularity at the separation line is attracting, repelling, or a center; analogously, Proposition \ref{prop:taylor-inf} can provide whether the infinity is attracting, repelling, or a center in the case it is monodromic; finally, Propositions \ref{prop:serietaylor} and \ref{prop:serieNewton-Puiseux} can be used to study the stability of some fold-fold connections. Mixing the stability properties above, one can immediately get sufficient conditions for the existence of a limit cycle (forcing, for instance, the mondromic singularity at the discontinuity line and the infinity, in the monodromic case, to have the same stability).

\section{Conclusions}

In this paper we provided fundamental properties of Poincar\'e half-maps defined on a straight line for planar linear systems. Our analysis was based on a novel characterization of Poincar\'e half-maps \cite{CARMONA2021319}, presented in Section \ref{sec:char}. This characterization has  proven to be an effective method to study these maps from a common point of view and to obtain results in a simple way, without the need of making particularized case-by-case studies. 

We have focused on the analyticity of the Poincar\'e half-maps, their series expansions, the relative position between the graph of Poincar\'e half-maps and the bisector of the fourth quadrant, and the sign of their second derivatives. In what follows, we sumarize the obtained results.
In Section \ref{ssec:orin-inf}, we addressed the series expansion of a Poincar\'e half-map, $P: I\subset [0,+\infty) \longrightarrow(-\infty,0]$, around the extrema of its interval of definition $I$, namely: Propositions \ref{prop:sequ} and \ref{prop:serietaylor} provided the Taylor expansion of  $P$ around  the origin when $0\in I$;  Proposition \ref{prop:serieNewton-Puiseux} provided the Newton-Puiseux series expansion of $P$ around $\hat{y}_0\in I$, where  $\hat{y}_0>0$ satisfies $P\left(\hat{y}_0\right)=0$; and Proposition \ref{prop:taylor-inf} provided the Taylor series expansions of  $P$ around the infinity when $4D-T^2>0$. In Section \ref{ssec:relativeposition}, Proposition \ref{rm:signoy0+y1} and Corollary \ref{rm:bi2quad}  established a relationship between the graphs of Poincar\'e half-maps and the bisector of the fourth quadrant  depending only  on the sign of the trace $T$. Finally, in Section \ref{sec:derivadasegunda}, Proposition \ref{prop:derivadasprimeraysegunda} provided expressions for the first and second derivative of the Poincar\'e half-maps and Proposition \ref{prop:signderivadasegunda} determined the sign of the second derivative of the Poincar\'{e} half-maps.

All these properties are essential to understand the dynamic behavior of planar piecewise linear systems with two zones separated by a straight line (PPWLS, for short). Thus, in Section \ref{sec:consequences} we provided some immediate consequences regarding periodic behavior of such systems, namely: Corollary \ref{cor:continuum} established non-generic conditions for a PPWLS either not having periodic orbits and having a continuum of  crossing periodic orbits; Corollary \ref{cor:TLTR>=0} gives generic conditions for a PPWLS not having periodic orbits; finally, Corollary \ref{cor:atmost2} provided generic conditions for a PPWLS having at most two limit cycles.

The results obtained in this paper also allow deeper insights regarding periodic solutions for piecewise linear systems. For instance, in \cite{CFN22}, the present results among others were of assistance in proving that PPWLS without sliding region (that is, $b=0$) has at most one limit cycles. This result was obtained without unnecessary distinctions of spectra of the matrices. In addition, it is proven that this limit cycle, if exists, is hyperbolic and its stability is determined by a simple condition in terms of the parameters. Also, in \cite{CarmonaEtAl23}, it was provide the existence of a uniform upper bound, $L^*$, for the maximum number of limit cycles of PPWLS. The present Proposition \ref{rm:signoy0+y1} helped to show that $L^*\leq  8$.

\section*{Acknowledgements}

VC and EGM are partially supported by the Ministerio de Ciencia, Innovaci\'on y Universidades, Plan Nacional I+D+I cofinanced with FEDER funds, in the frame of the project PGC2018-096265-B-I00. FFS is partially supported by the Ministerio de Econom\'{i}a y Competitividad, Plan Nacional I+D+I cofinanced with FEDER funds, in the frame of the project MTM2017-87915-C2-1-P. VC, FFS, and EGM are partially supported by the Ministerio de Ciencia e Innovaci\'on, Plan Nacional I+D+I cofinanced with FEDER funds, in the frame of the project PID2021-123200NB-I00, the Consejer\'{i}a de Educaci\'{o}n y Ciencia de la Junta de Andaluc\'{i}a (TIC-0130, P12-FQM-1658) and by the Consejer\'{i}a de Econom\'{i}a, Conocimiento, Empresas y Universidad de la Junta de Andaluc\'{i}a (US-1380740, P20-01160). DDN is partially supported by S\~{a}o Paulo Research Foundation (FAPESP) grants 2022/09633-5, 2021/10606-0, 2018/13481-0, and 2019/10269-3, and by Conselho Nacional de Desenvolvimento Cient\'{i}fico e Tecnol\'{o}gico (CNPq) grants 438975/2018-9 and 309110/2021-1.

\section{Statements and Declarations}

\noindent Data Availability: Data sharing not applicable to this article as no datasets were generated or analysed during the current study.

\bigskip

\noindent Conflict of Interest: The authors have no conflicts of interest to declare that are relevant to the content of this article.

\bibliographystyle{abbrv}
\bibliography{references.bib}

\begin{thebibliography}{10}

\bibitem{CANO2020}
J.~Cano, S.~Falkensteiner, and J.~R. Sendra.
\newblock Existence and convergence of puiseux series solutions for autonomous
  first order differential equations.
\newblock {\em Journal of Symbolic Computation}, 2020.

\bibitem{CARMONA2021319}
V.~Carmona and F.~Fern{\'a}ndez-S{\'a}nchez.
\newblock Integral characterization for {P}oincar{\'e} half-maps in planar
  linear systems.
\newblock {\em Journal of Differential Equations}, 305:319--346, 2021.

\bibitem{Carmona21}
V.~Carmona, F.~Fern{\'a}ndez-S{\'a}nchez, and D.~D. Novaes.
\newblock A new simple proof for lum--chua's conjecture.
\newblock {\em Nonlinear Analysis: Hybrid Systems}, 40:100992, 2021.

\bibitem{CFN22}
V.~Carmona, F.~Fernández-Sánchez, and D.~D. Novaes.
\newblock Uniqueness and stability of limit cycles in planar piecewise linear
  differential systems without sliding region, 2022.
\newblock arXiv:2207.14634.

\bibitem{CarmonaEtAl23}
V.~Carmona, F.~Fernández-Sánchez, and D.~D. Novaes.
\newblock Uniform upper bound for the number of limit cycles of planar
  piecewise linear differential systems with two zones separated by a straight
  line.
\newblock {\em Applied Mathematics Letters}, 137:108501, 2023.

\bibitem{CarmonaEtAl02}
V.~Carmona, E.~Freire, E.~Ponce, and F.~Torres.
\newblock On simplifying and classifying piecewise-linear systems.
\newblock {\em IEEE Trans. Circuits Systems I Fund. Theory Appl.},
  49(5):609--620, 2002.

\bibitem{Chicone82}
C.~Chicone.
\newblock Bifurcations of nonlinear oscillations and frequency entrainment near
  resonance.
\newblock {\em SIAM J. Math. Anal.}, 23(6):1577--1608, 1982.

\bibitem{COLL2001671}
B.~Coll, A.~Gasull, and R.~Prohens.
\newblock Degenerate hopf bifurcations in discontinuous planar systems.
\newblock {\em Journal of Mathematical Analysis and Applications},
  253(2):671--690, 2001.

\bibitem{FreireEtAl98}
E.~Freire, E.~Ponce, F.~Rodrigo, and F.~Torres.
\newblock Bifurcation sets of continuous piecewise linear systems with two
  zones.
\newblock {\em Internat. J. Bifur. Chaos Appl. Sci. Engrg.}, 8(11):2073--2097,
  1998.

\bibitem{FreireEtAl12}
E.~Freire, E.~Ponce, and F.~Torres.
\newblock Canonical discontinuous planar piecewise linear systems.
\newblock {\em SIAM J. Appl. Dyn. Syst.}, 11(1):181--211, 2012.

\bibitem{FreireEtAl2013}
E.~Freire, E.~Ponce, and F.~Torres.
\newblock Planar {F}ilippov systems with maximal crossing set and piecewise
  linear focus dynamics.
\newblock In {\em Progress and Challenges in Dynamical Systems}, pages
  221--232. Springer Berlin Heidelberg, 2013.

\bibitem{Garcia-Barroso:2017ty}
E.~R. Garc{\'\i}a~Barroso, P.~D. Gonz{\'a}lez~P{\'e}rez, and P.~Popescu-Pampu.
\newblock Variations on inversion theorems for {N}ewton--{P}uiseux series.
\newblock {\em Mathematische Annalen}, 368(3):1359--1397, 2017.

\bibitem{HanZhang10}
M.~Han and W.~Zhang.
\newblock On {H}opf bifurcation in non-smooth planar systems.
\newblock {\em J. Differential Equations}, 248(9):2399--2416, 2010.

\bibitem{henrici}
P.~Henrici.
\newblock {\em Applied and computational complex analysis. {V}ol. 1}.
\newblock Wiley Classics Library. John Wiley \& Sons, Inc., New York, 1988.
\newblock Power series---integration---conformal mapping---location of zeros,
  Reprint of the 1974 original, A Wiley-Interscience Publication.

\bibitem{Huan:2021ti}
S.-M. Huan.
\newblock On the number of limit cycles in general planar piecewise linear
  differential systems with two zones having two real equilibria.
\newblock {\em Qualitative Theory of Dynamical Systems}, 20(1):4, 2021.

\bibitem{HuanYang2013}
S.-M. Huan and X.-S. Yang.
\newblock Existence of limit cycles in general planar piecewise linear systems
  of saddle{\textendash}saddle dynamics.
\newblock {\em Nonlinear Analysis: Theory, Methods {\&} Applications},
  92:82--95, Nov. 2013.

\bibitem{HuanYang2014}
S.-M. Huan and X.-S. Yang.
\newblock On the number of limit cycles in general planar piecewise linear
  systems of node{\textendash}node types.
\newblock {\em Journal of Mathematical Analysis and Applications},
  411(1):340--353, Mar. 2014.

\bibitem{LI2021101045}
S.~Li, C.~Liu, and J.~Llibre.
\newblock The planar discontinuous piecewise linear refracting systems have at
  most one limit cycle.
\newblock {\em Nonlinear Analysis: Hybrid Systems}, 41:101045, 2021.

\bibitem{LiLLibre2020}
S.~Li and J.~Llibre.
\newblock Phase portraits of planar piecewise linear refracting systems:
  Focus-saddle case.
\newblock {\em Nonlinear Analysis: Real World Applications}, 56:103153, 2020.

\bibitem{LlibreEtAl15}
J.~Llibre, D.~D. Novaes, and M.~A. Teixeira.
\newblock Maximum number of limit cycles for certain piecewise linear dynamical
  systems.
\newblock {\em Nonlinear Dynam.}, 82(3):1159--1175, 2015.

\bibitem{LlibreEtAl2008}
J.~Llibre, E.~Ponce, and F.~Torres.
\newblock On the existence and uniqueness of limit cycles in {L}i{\'{e}}nard
  differential equations allowing discontinuities.
\newblock {\em Nonlinearity}, 21(9):2121--2142, Aug. 2008.

\bibitem{LT14}
J.~Llibre and A.~E. Teruel.
\newblock {\em Introduction to the qualitative theory of differential systems}.
\newblock Birkh\"{a}user Advanced Texts: Basler Lehrb\"{u}cher. [Birkh\"{a}user
  Advanced Texts: Basel Textbooks]. Birkh\"{a}user/Springer, Basel, 2014.
\newblock Planar, symmetric and continuous piecewise linear systems.

\bibitem{MedradoTorregrosa15}
J.~C. Medrado and J.~Torregrosa.
\newblock Uniqueness of limit cycles for sewing planar piecewise linear
  systems.
\newblock {\em J. Math. Anal. Appl.}, 431(1):529--544, 2015.

\end{thebibliography}

\end{document}